\documentclass[12pt,reqno]{amsart}
\usepackage{enumerate, latexsym, amsmath, amsfonts, amssymb, amsthm, color}

\usepackage{enumitem}[2011/09/28]
\setenumerate{align=left, leftmargin=0pt,labelsep=.5em, labelindent=0\parindent,listparindent=\parindent,itemindent=*}

\def\qed{\hfill$ \blacksquare$}
\def\eor{\hfill$ \square$}
\DeclareMathOperator{\D}{d}

\DeclareMathOperator{\disc}{disc}
\DeclareMathOperator{\I}{Im}

\DeclareMathOperator{\RE}{Re}
\def\pmod #1{\ ({\rm{mod}}\ #1)}

\def\l{\left}
\def\r{\right}
\def\bg{\bigg}
\def\({\bg(}
\def\){\bg)}

\def\f{\frac}

\def\bi{\binom}

\theoremstyle{plain}
\newtheorem{theorem}{Theorem}

\newtheorem{lemma}{Lemma}

\theoremstyle{definition}

\theoremstyle{remark}
\newtheorem{remark}{Remark}

\makeatletter
\@namedef{subjclassname@2020}{%
  \textup{2020} Mathematics Subject Classification}
\makeatother
 \vspace{4mm}

\usepackage{colonequals, mathrsfs, url}
\numberwithin{equation}{section}
\begin{document}

\setcounter{lemma}{0}
\setcounter{theorem}{0}
\setcounter{corollary}{0}
\setcounter{remark}{0}
\setcounter{equation}{0}

\hbox{}
\medskip
\title
{Fast converging irrational series for $ L(2,(\frac d\cdot))$}

\author{Zhi-Wei Sun and Yajun Zhou}

\address[Zhi-Wei Sun]{School of Mathematics, Nanjing
University, Nanjing 210093, People's Republic of China}
\email{zwsun@nju.edu.cn}

\address[Yajun Zhou]{Program in Applied and Computational Mathematics (PACM), Princeton University, Princeton, NJ 08544} \email{yajunz@math.princeton.edu}\curraddr{\textrm{} \textsc{Academy of Advanced Interdisciplinary Studies (AAIS), Peking University, Beijing 100871, P. R. China}}\email{yajun.zhou.1982@pku.edu.cn}

\keywords{Dirichlet $L$-functions, modular forms, series for $\pi$, binomial coefficients.
\newline \indent The first author is supported by the Natural Science Foundation of China (grant 12371004).}

\begin{abstract}
By exploring the theory of Guillera--Rogers, we evaluate some infinite series whose summands are quadratic irrationals,
in terms of $\pi$ and special values of Dirichlet $L$-functions $ L_d(2)\equiv L(2,(\frac d\cdot)):=\sum_{k=1}^\infty\left( \frac{d}{k} \right)\frac1{k^2}$. Applying Kronecker's theorem to linear combinations of lattice sums, we obtain geometrically convergent series for $ L_{-56}(2)$, $ L_{-68}(2)$, $ L_{-87}(2)$, $ L_{-111}(2)$, and $ L_{-116}(2)$, which go beyond the solvable cases of Guillera--Rogers.\end{abstract}
\maketitle

\medskip
\section{Introduction}Write $ \binom{n}{k}\colonequals \frac{n!}{k!(n-k)!}$ for binomial coefficients. The so-called Ramanujan series upside-down\begin{align}
\sum_{k=1}^\infty\frac{(ak-b)m^k}{k^{3}\binom{2k}{k}^3}\label{eq:RamaUD}
\end{align}  involving algebraic parameters $a,b,m\in\overline{\mathbb Q} $  were investigated  by Guillera--Rogers \cite{GuilleraRogers2014}. The first identity of this kind\begin{align}
\sum_{k=1}^\infty\frac{21k-8}{k^{3}\binom{2k}{k}^3}=\frac{\pi^{2}}{6}
\end{align}was given by Zeilberger \cite{Zeilberger1993pun} via the WZ method.

In this note, we present some  closed-form evaluations for Ramanujan series upside-down that were not found in the original work of Guillera--Rogers \cite{GuilleraRogers2014}.

In \S\ref{sec:FibLucas}, we deal with irrational series that evaluate to rational multiples of $\pi^2$, as stated in the theorem below.

\begin{theorem}\label{thm:FibLucas}We have \begin{align}
\sum_{k=1}^\infty\frac{3 \left(16 \sqrt{5}-35\right) k-4 \left(5 \sqrt{5}-11\right)}{k^{3}\binom{2k}{k}^3}\left(\frac{1+\sqrt{5}}{2} \right)^{8 k}={}&\frac{71\pi^{2}}{30},\label{eq:GRnew}
\intertext{which is dual to a result of Guillera--Rogers \cite[(65)]{GuilleraRogers2014}:}
\sum_{k=1}^\infty\frac{3 \left(16 \sqrt{5}+35\right) k-4 \left(5 \sqrt{5}+11\right)}{k^{3}\binom{2k}{k}^3}\left(\frac{1-\sqrt{5}}{2} \right)^{8 k}={}&\frac{\pi^{2}}{30}.\label{eq:GRold}
\end{align}\end{theorem}\begin{remark}One can define the $n$-th Fibonacci number $ F_n$ and the  $n$-th  Lucas number $ L_n$  recursively by\begin{align}
\left\{\begin{array}{@{}lll}
F_0=0, & F_1=1, & F_{n+2} =F_{n+1}+F_n;\\
L_{0}=2, & L_1=1, & L_{n+2} =L_{n+1}+L_n. \\
\end{array}\right.
\end{align}By $ \mathbb Q(\sqrt{5})$-linear combinations of \eqref{eq:GRnew} and \eqref{eq:GRold}, we may deduce \begin{align}
\sum_{k=1}^{\infty}\frac{20 (12 k-5) F_{8 k}-(105 k-44) L_{8 k}}{k^{3}\binom{2k}{k}^3}={}&\frac{7\pi^{2}}{3},\label{eq:FibLucas1}\\\sum_{k=1}^{\infty}\frac{4 (12 k-5) L_{8 k}-(105 k-44) F_{8 k}}{k^{3}\binom{2k}{k}^3}={}&\frac{12 \pi ^2}{5 \sqrt{5}}.\label{eq:FibLucas2}
\end{align} Expressing the  Lucas numbers in terms of Fibonacci numbers:\begin{align}
L_{n}={}&2F_{n+1}-F_n=2(F_n+F_{n-1})-F_n=F_n+2F_{n-1},
\end{align} we may also transcribe  \eqref{eq:FibLucas1}--\eqref{eq:FibLucas2}  into
\begin{align}
\sum_{k=1}^{\infty}\frac{(135 k-56) F_{8 k}-2 (105 k-44) F_{8 k-1}}{k^{3}\binom{2k}{k}^3}={}&\frac{7\pi^{2}}{3},\label{eq:FibLucas1'}\tag{\ref{eq:FibLucas1}$'$}\\\sum_{k=1}^{\infty}\frac{8 (12 k-5) F_{8 k-1}-3 (19 k-8) F_{8 k}}{k^{3}\binom{2k}{k}^3}={}&\frac{12 \pi ^2}{5 \sqrt{5}},\label{eq:FibLucas2'}\tag{\ref{eq:FibLucas2}$'$}
\end{align}without invoking the Lucas numbers in the summands.
\eor\end{remark}
In \S\ref{sec:specL2}, we prove the next theorem motivated by Z.-W.\ Sun \cite{Sun2025IrrSeries}, where irrational series are expressed  in terms of the Dirichlet $L $-values $ L_d(2)\colonequals \sum_{k=1}^\infty\left( \frac{d}{k} \right)\frac1{k^2}$, with    $ \left(\frac{d}{\cdot}\right)$ being the   Kronecker symbol. Examples of such special Dirichlet $L $-values  include \begin{align}
K\equiv L_{-3}(2)\colonequals\sum_{k=1}^\infty\left( \frac{-3}{k} \right)\frac{1}{k^{2}}= \sum_{k=0}^\infty\left[\frac{1}{(3k+1)^2}-\frac{1}{(3k+2)^2}\right],
\end{align}\begin{align} G\equiv L_{-4}(2)\colonequals\sum_{k=1}^\infty\left( \frac{-4}{k} \right)\frac{1}{k^{2}}= \sum_{k=0}^\infty\frac{(-1)^k}{(2k+1)^2},\end{align}
and \begin{align} L_{-8}(2)\colonequals \sum_{k=1}^\infty\left( \frac{-8}{k} \right)\frac{1}{k^{2}}=\sum_{k=0}^\infty\frac{(-1)^{k(k-1)/2}}{(2k+1)^2}.\end{align} \begin{theorem}\label{thm:2k3k4kL2}\begin{enumerate}[leftmargin=*,
label=\emph{(\alph*)},ref=(\alph*),
widest=d, align=left] \item\label{itm:2kL2}
We have \begin{align}\begin{split}&
\sum_{k=1}^\infty
\frac{132 \left(5 \sqrt{2}+7\right) k-284 \sqrt{2}-401}{k^3 \binom{2 k}{k}^3}\left[ -\big( 2-\sqrt{2} \big)^{12} \right]^k\\={}&16\left[ 23G-14\sqrt{2}L_{-8}(2) \right],
\end{split}\label{eq:d=-352}
\end{align}\begin{align}
\begin{split}&
\sum_{k=1}^\infty
\frac{132 \left(70 \sqrt{29}+377\right) k-4234\sqrt{29}-22801}{k^3 \binom{2 k}{k}^3}\left[ -\big( 5\sqrt{29} -27\big)^{6} \right]^k\\={}&\frac{8}{3}\left[ 29\sqrt{29} L_{-116}(2)-198G\right],
\end{split}\label{eq:d=-928}
\end{align}\begin{align}
\begin{split}
&\sum_{k=1}^\infty\frac{6 \left(17 \sqrt{7}+35\right) k- 35 \sqrt{7}-89}{k^{3}\binom{2k}{k}^3}\left(-2^{11}\right)^k\big(45-17\sqrt{7}\big)^{2k}\\={}&128\left[ 20L_{-8}(2)-7\sqrt{7}L_{-56}(2) \right],\label{eq:d=-448}
\end{split}\end{align}\begin{align}
\begin{split}&
\sum_{k=1}^\infty
\frac{3 \left(85 \sqrt{7}+224\right) k-8\left(14 \sqrt{7}+37\right)}{k^3 \binom{2 k}{k}^3}(3\sqrt{7}-8)^{3k}\\={}&32G-\frac{77\sqrt{7}}{8}L_{-7}(2).
\end{split}\label{eq:d=-112}
\end{align}\item \label{itm:3kL2}We have
\begin{equation}\sum_{k=1}^\infty\f{(3\sqrt3+7)k-\sqrt3-2}{k^3\bi{2k}k^2\bi{3k}k}
[24(12-7\sqrt3)]^k=\f56\l(9\sqrt3 K-16G\r),
\end{equation}
\begin{equation}\begin{aligned}&\sum_{k=1}^\infty\f{3(17\sqrt3+27)k-16\sqrt3-27}{k^3\bi{2k}k^2\bi{3k}k}
[54(265-153\sqrt3)]^k
\\={}&135\l(G-\f{11}{16}\sqrt3 K\r),
\end{aligned}\end{equation}
\begin{equation}\begin{aligned}&\sum_{k=1}^\infty\f{9(11\sqrt6+51)k-2(19\sqrt6+54)}{k^3\bi{2k}k^2\bi{3k}k}
[27(37102-15147\sqrt6)]^k
\\={}&\f{5625}2\l[\sqrt6 L_{-24}(2)-3G\r],
\end{aligned}
\end{equation}
\begin{equation}\begin{aligned}&\sum_{k=1}^\infty\f{15(9\sqrt{13}+26)k-39\sqrt{13}-134}{k^3\bi{2k}k^2\bi{3k}k}
(-1728)^k(18-5\sqrt{13})^{2k}
\\={}&54\l[80K-13\sqrt{13}L_{-39}(2)\r],
\end{aligned}
\end{equation}
\begin{equation}\begin{aligned}&\sum_{k=1}^\infty\f{(91\sqrt{33}+891)k-33\sqrt{33}-225}{k^3\bi{2k}k^2\bi{3k}k}
[32(91\sqrt{33}-523)]^{k}
\\={}&320\l[\f{11}3\sqrt{33}L_{-11}(2)-27K\r],
\end{aligned}
\end{equation}
\begin{equation}\begin{aligned}&\sum_{k=1}^\infty\f{15k(2898\sqrt{37}+3145)k-6438\sqrt{37}-30355}
{k^3\bi{2k}k^2\bi{3k}k}
\\&{}\times(-3)^k(24(2737-450\sqrt{37}))^{2k}\\={}&71874\l[380K-37\sqrt{37}L_{-111}(2)\r],
\end{aligned}
\end{equation}
and
\begin{equation}\begin{aligned}&\sum_{k=1}^\infty\f{15(323\sqrt{145}+13195)k-2407\sqrt{145}-35375}{k^3\bi{2k}k^2\bi{3k}k}
\\{}&\times (-3/5)^k[12(2975-247\sqrt{145})]^{2k}\\={}&23328\l[29\sqrt{145}L_{-87}(2)-375L_{-15}(2)\r].
\end{aligned}
\end{equation}

 \item \label{itm:4kL2}We have
\begin{equation}\begin{aligned}&\sum_{k=1}^\infty\f{5(13\sqrt2+32)k-3(6\sqrt2+11)}{k^3\bi{2k}k^2\bi{4k}{2k}}
[8(457-325\sqrt2)]^k
\\ ={}&490\sqrt2L_{-8}(2)-980G,
\end{aligned}
\end{equation}
\begin{equation}\begin{aligned}&\sum_{k=1}^\infty\f{
5980(11092\sqrt{11}+19437)k-11937508\sqrt{11}-37515813}{k^3\bi{2k}k^2\bi{4k}{2k}}
\\{}&\times(-4)^k(84(1121-338\sqrt{11}))^{4k}\\  ={}&36120^2\l[36G-11\sqrt{11}L_{-11}(2)\r],
\end{aligned}\label{eq:d=-1012}
\end{equation}
and\begin{equation}\begin{aligned}&\sum_{k=1}^\infty\f{140(92\sqrt{17}+2091)k-6868\sqrt{17}-36591}{k^3\bi{2k}k^2\bi{4k}{2k}}
\\{}&\times(-4)^k(12(41-10\sqrt{17}))^{4k}
\\ ={}&415812\left[17\sqrt{17}L_{-68}(2)-90G\right].
\end{aligned}\label{eq:d=-340}
\end{equation}

\end{enumerate}\end{theorem}
\begin{remark}Five rational series similar to those in parts (b) and (c) of Theorem \ref{thm:2k3k4kL2} were initially conjectured by Sun \cite[Conjecture 1.4]{Sun2011}. Their proofs can be found in \cite{GuilleraRogers2014}.\eor\end{remark}
\section{Proof of Theorem \ref{thm:FibLucas}\label{sec:FibLucas}}

\subsection{Evaluations of irrational series via rational multiples of $ \pi^2$}

\subsubsection{Guillera--Rogers theory revisited}
Before recalling a reformulation \cite[Theorem 4.3]{Zhou2025BHS}    of the  Guillera--Rogers theory  in Lemma \ref{lm:GR} below, we need to fix some notations.

Along  with the  Legendre function\begin{align}
P_\nu(1-2t)\colonequals-\frac{\sin(\nu\pi)}{\pi}\int_0^{1}\left[ \frac{X(1-tX)}{1-X} \right]^\nu\frac{\D X}{1-X}
\end{align} of degrees $ \nu\in\left\{ -\frac{1}{4} ,-\frac{1}{3},-\frac{1}{2}\right\}$ defined for $ t\in\mathbb C\smallsetminus[1,\infty)$, we will employ the Legendre--Ramanujan functions \cite[(2.2.11)]{AGF_PartI}
\begin{align}\begin{split}&
R_{\nu}( \xi)\\:={}&\frac{1-\xi^{2}}{P_{\nu}(\xi)}\frac{\D P_{\nu}(\xi)}{\D \xi}+\frac{1-\xi^{2}}{P_{\nu}(-\xi)}\frac{\D P_{\nu}(-\xi)}{\D \xi}\\{}&-\frac{\sin(\nu\pi)}{\pi}\left\{\frac{1}{ [P_{\nu}(\xi)]^{2}\I \frac{iP_{\nu}(-\xi)}{P_{\nu}(\xi)}}-\frac{1}{[P_{\nu}(-\xi)]^{2}\I \frac{iP_{\nu}(\xi)}{P_{\nu}(-\xi)}}\right\}\end{split}\end{align}that are defined for  $\xi\in(\mathbb C\smallsetminus\mathbb R)\cup(-1,1)$ and admit continuous extensions to all $ \xi\in\mathbb C$.  For $ N=4\sin^2(\nu\pi)\in\{2,3,4\}$, introduce the modular invariants\begin{align}
\begin{split}
\alpha_N(z)\colonequals{}&
\left\{1+\frac{1}{N^{6/(N-1)}}\left[ \frac{\eta(z)}{\eta(Nz)} \right]^{24/(N-1)}\right\}^{-1}\\={}&1-\alpha_N\left( -\frac{1}{N z} \right)
\end{split}
\end{align}via the Dedekind eta function $\eta(z)\colonequals e^{\pi iz/12}\prod_{n=1}^\infty(1-e^{2\pi inz})$, where $ z\in\mathfrak H\colonequals \{w\in\mathbb C|\I w>0\}$.  Here, if $z$ is a CM point (namely, subject to an additional constraint that $ [\mathbb Q(z):\mathbb Q]=2$), then  $ \alpha_{N}(z),R_{\nu}(1-2\alpha_{N}(z))\in\overline{\mathbb Q}\cup\{\infty\}$, and there are explicit algorithms to compute these algebraic numbers in question  \cite[\S6]{Zagier2008Mod123}.

With the notations\begin{align}
\begin{split}&
\mathsf\Sigma_{-1/4}^{GR}(z)\\\colonequals {}&\sum_{k=1}^\infty\frac{2[1-2\alpha_2(z)]k-R_{-1/4}(1-2\alpha_2(z))}{k^3\binom{2k}k^{2}\binom{4k}{2k}}\left\{ \frac{64}{\alpha_2(z)[1-\alpha_{2}(z)]} \right\}^{k},
\end{split}
\end{align} \begin{align}
\begin{split}&
\mathsf\Sigma_{-1/3}^{GR}(z)\\\colonequals {}&\sum_{k=1}^\infty\frac{2[1-2\alpha_3(z)]k-R_{-1/3}(1-2\alpha_3(z))}{k^3\binom{2k}k^{2}\binom{3k}{k}}\left\{ \frac{27}{\alpha_3(z)[1-\alpha_{3}(z)]} \right\}^{k},
\end{split}
\end{align}and\begin{align}
\begin{split}&
\mathsf\Sigma_{-1/2}^{GR}(z)\\\colonequals {}&\sum_{k=1}^\infty\frac{2[1-2\alpha_4(z)]k-R_{-1/2}(1-2\alpha_4(z))}{k^3\binom{2k}k^{3}}\left\{ \frac{16}{\alpha_4(z)[1-\alpha_{4}(z)]} \right\}^{k},
\end{split}
\end{align}as well as the correspondence  $ N=4\sin^2(\nu\pi)\in\{2,3,4\}$ for  $ \nu\in\left\{ -\frac{1}{4} ,-\frac{1}{3},-\frac{1}{2}\right\}$, we may restate part of  \cite[Theorem 4.3]{Zhou2025BHS}   in the lemma below.\begin{lemma}\label{lm:GR}
For $z\in\mathfrak{H}$ satisfying~\begin{align}
\left\{\begin{array}{@{}l}
|4\alpha_N(z)[1-\alpha_{N}(z)]|\geq1, \alpha_{N}(z)\neq\frac12, \\
|\RE z|\leq\frac12,\big|z+\frac1N\big|\geq\frac1N,\big|z-\frac1N\big|\geq\frac1N, \\
\end{array}\right.
\label{ineq:z}\end{align}     we have

\begin{align}
\begin{split}&\I\mathsf\Sigma_{\nu}^{GR}(z)
\\={}&\frac{4\pi^{2}}{3\I z}\left( \RE z-\frac{\RE z}{2|\RE z|} \right)\left[ \left( \RE z-\frac{\RE z}{2|\RE z|} \right)^2+3(\I z)^2+\frac{12-N}{4N} \right]\\{}&+\frac{4\pi^{2}\RE\Big[  \widetilde{\mathscr E}_4(Nz)-N\widetilde{\mathscr E}_4(z) \Big]}{N(N^{2}-1)\I z},\end{split}\label{eq:GR}
\end{align}with     \begin{align} \widetilde{\mathscr E}_4(z)\colonequals\int_z^{i\infty}[1-E_4(w)](z-w)(\overline z-w)\D w\label{eq:EichlerE4}\end{align}defined through the Eisenstein series\begin{align}
E_4(z)\colonequals \displaystyle 1+240\sum_{n=1}^\infty\dfrac{n^3e^{2\pi inz}}{1-e^{2\pi inz}},\quad z\in\mathfrak H.
\end{align}Here, one has (cf.\  \cite[Proposition 3]{GuilleraRogers2014})\begin{align}\begin{split}
&\RE\widetilde{\mathscr E}_4(z)\\={}&\begin{cases}0 &\text{if } 2\RE z\in\mathbb Z, \\
-\frac{\RE z}{3}\left[ \frac{|z|^{2}+2(\I z)^2}{|z|^{4}} +|z|^{2}+2(\I z)^2-5\right] & \text{if }2\RE \frac{1}{z}\in\mathbb Z, \\
\end{cases}
\end{split}\label{eq:GR_ReE}
\end{align}for all $ z\in\mathfrak H$.\qed\end{lemma}\begin{remark}One can derive the second entry on the right-hand side  of   \eqref{eq:GR_ReE} from the following functional equation     \cite[(4.53)]{Zhou2025BHS}\begin{align}\begin{split}&
\frac{|z|^{2}}{(\I z)^2}\RE\widetilde {\mathscr E}_4\left( -\frac{1}{z} \right)-
\frac{\RE\widetilde {\mathscr E}_4(z)}{(\I z)^2}\\={}&\frac{\RE z}{3(\I z)^2}\left[ \frac{|z|^{2}+2(\I z)^2}{|z|^{4}} +|z|^{2}+2(\I z)^2-5\right],\quad  z\in\mathfrak H,
\end{split}\label{eq:ReE4_refl}
\end{align} which descends from Ramanujan's reflection formula \cite[p.\ 276]{RN2}\begin{align}
\mathscr E_4(z)-z^2\mathscr E_4\left( -\frac{1}{z} \right)=-\frac{z^4-5 z^2+1}{3 z}+\frac{30 i\zeta (3) \left(z^2-1\right) }{\pi ^3}
\end{align} for    $\mathscr E_4(z)\colonequals\int_{z}^{i\infty} [1-E_4(w)](z-w)^2\D w=\frac{60 i}{\pi^3}\sum_{k=1}^\infty\frac{1}{k^3(e^{-2\pi ik z}-1)}$.
\eor\end{remark}\begin{remark}We may compute the following special values{\allowdisplaybreaks
\begin{align}
2i[1-2\alpha_{4}(z)]\I z={}&\begin{cases}-\frac{3}{4} \left(16 \sqrt{5}-35\right) & \text{if } z=\frac{-1+\sqrt{15}i}{8}, \\
-\frac{3}{8}\left(16 \sqrt{5}+35\right) &\text{if } z=\frac{-7+\sqrt{15}i}{16}, \\
\end{cases}\\-iR_{-1/2}(1-2\alpha_{4}(z))\I z={}&\begin{cases}5 \sqrt{5}-11 & \text{if } z=\frac{-1+\sqrt{15}i}{8}, \\
\frac{5 \sqrt{5}+11}{2} &\text{if } z=\frac{-7+\sqrt{15}i}{16}, \\
\end{cases}
\end{align}
}from explicit algorithms   \cite[\S6]{Zagier2008Mod123}. Here, both $z=\frac{-1+\sqrt{15}i}{8}$ and $ z=\frac{-7+\sqrt{15}i}{16}$ fit into \eqref{ineq:z} for $N=4$.
\eor\end{remark}With the lemma above, one  immediately arrives at the next theorem.\begin{theorem}\label{thm:ImGR}If  $z$ is a CM point satisfying \eqref{ineq:z} and  $\RE\big[   \widetilde{\mathscr E}_4(Nz)-N\widetilde{\mathscr E}_4(z)\big]$ can be computed from \eqref{eq:GR_ReE} , then\begin{align}
\I\mathsf\Sigma_{\nu}^{GR}(z)\label{eq:ImGR_N}
\end{align} is a rational multiple of $ \pi^2$, and all the summands of the series above are explicitly computable algebraic numbers.
\qed
\end{theorem} \subsubsection{Proof of \eqref{eq:GRnew}}
With $z=\frac{-1+\sqrt{15}i}{8}$, we have \begin{align}
2\RE(4z)=-1\in\mathbb Z,\quad 2\RE\frac{1}{z}=-1\in\mathbb Z.
\end{align}Therefore, the right-hand side of \eqref{eq:GR} reduces to \begin{align}
\begin{split}&
\left\{\frac{4\pi^{2}}{3\I z}\left( \RE z-\frac{\RE z}{2|\RE z|} \right)\left[ \left( \RE z-\frac{\RE z}{2|\RE z|} \right)^2+3(\I z)^2+\frac{1}{2} \right]\right.\\&{}\left.\left.{}+\frac{4\pi^{2}\RE z}{45\I z}\left[ \frac{|z|^{2}+2(\I z)^2}{|z|^{4}} +|z|^{2}+2(\I z)^2-5\right]\right\}\right|_{z=\frac{-1+\sqrt{15}i}{8}}\\={}&\frac{71 \pi ^2}{15 \sqrt{15}},
\end{split}
\end{align}as expected.
\subsubsection{Proof of \eqref{eq:GRold} }The series evaluation \eqref{eq:GRold}  originally appeared as \cite[(65)]{GuilleraRogers2014}, and drew on a special CM\ point $z'= \frac{9+\sqrt{15}i}{16}=\frac{-7+\sqrt{15}i}{16}+1=z+1$ that failed the criterion $ |\RE z'|<\frac12$ [cf.\ \eqref{ineq:z}].

However, by choosing the CM\ point $z'= \frac{9+\sqrt{15}i}{16}$, Guillera and Rogers gained  a relation  $ 2\RE\frac1{z'}=2\RE\big( \frac{3}{2}-\frac{\sqrt{15}i}{6}   \big)\in\mathbb Z$. Consequently, \begin{align}
\begin{split}
\RE\widetilde{\mathscr E}_4(z')={}&-\frac{\RE z'}{3}\left[ \frac{|z'|^{2}+2(\I z')^2}{|z'|^{4}} +|z'|^{2}+2(\I z')^2-5\right]\\={}&\frac{387}{2048}
\end{split}
\end{align} was readily computable from \eqref{eq:GR_ReE}. By the translational invariance  $ E_4(w+1)=E_4(w)$, one has $ \RE\widetilde{\mathscr E}_4(z')=\RE\widetilde{\mathscr E}_4(z'-1)=\RE\widetilde{\mathscr E}_4(z)$, so\begin{align}
\begin{split}&
\RE\widetilde{\mathscr E}_4\left( \frac{-7+\sqrt{15}i}{16} \right)=\frac{387}{2048}\\={}&-\RE\widetilde{\mathscr E}_4\left( \frac{7+\sqrt{15}i}{16} \right)=-\RE\widetilde{\mathscr E}_4\left( -\frac{1}{4z} \right).
\end{split}
\end{align}Thus, one can evaluate $ \RE\widetilde{\mathscr E}_4\left( 4z \right)=-\frac{1}{32}$ by resorting to \eqref{eq:ReE4_refl}. Afterwards, we complete the right-hand side of \eqref{eq:GR}, arriving at a closed-form evaluation $ \I\Sigma_{-1/2}^{GR}\left( \frac{-7+\sqrt{15}i}{16} \right)=\frac{\pi ^2}{15 \sqrt{15}}$, as claimed.

\subsection{Series involving Fibonacci numbers and Lucas numbers\label{subsec:FL}}Now,  rewriting the summands of \eqref{eq:GRnew} and \eqref{eq:GRold} with\begin{align}
\left(\frac{1\pm\sqrt{5}}{2}\right)^{8k}=\frac{L_{8k}\pm \sqrt{5}F_{8k}}{2}
\end{align}as well as the observation that \begin{align}
\begin{split}
&\left[ 3\left( \pm 16\sqrt{5}-35 \right)k-4\left( \pm5\sqrt{5}-11 \right)\right]\left( L_{8k}\pm \sqrt{5}F_{8k} \right)\\={}&\left[ \pm(48k-20)\sqrt{5}+44-105k \right]\left( L_{8k}\pm \sqrt{5}F_{8k} \right)\\={}&\left( 44-105k \right)L_{8k}+5(48k-20)F_{8k}\pm\sqrt{5}[(44-105k)F_{8k}\\&{}+4(12k-5)L_{8k}],
\end{split}
\end{align}one arrives at\begin{align}\begin{split}
&
\sum_{k=1}^{\infty}\frac{20 (12 k-5) F_{8 k}-(105 k-44) L_{8 k}}{k^{3}\binom{2k}{k}^3}
\\&{}+(-1)^{m}\sqrt{5}\sum_{k=1}^{\infty}\frac{4 (12 k-5) L_{8 k}-(105 k-44) F_{8 k}}{k^{3}\binom{2k}{k}^3}\\={}&\begin{cases}\frac{71\pi^{2}}{15} & \text{if }m=0, \\
-\frac{\pi^{2}}{15} & \text{if }m=1. \\
\end{cases}
\end{split}\label{eq:FLpm}\end{align}Taking \eqref{eq:FLpm}$ _{m=0}+$\eqref{eq:FLpm}$ _{m=1}$ [resp.\ \eqref{eq:FLpm}$ _{m=0}-$\eqref{eq:FLpm}$_{m=1}$], we get \eqref{eq:FibLucas1}
 [resp.\ \eqref{eq:FibLucas2}].
\section{Proof of Theorem \ref{thm:2k3k4kL2}\label{sec:specL2}}

 \subsection{Evaluations of irrational series via Epstein zeta functions}
Fix a Hecke congruence group \begin{align} \varGamma_0(N):=\left\{ \left.\left(\begin{smallmatrix}a&b\\c&d\end{smallmatrix}\right)\right| a,b,c,d\in\mathbb Z; ad-bc=1;c\equiv0\pmod N\right\}\end{align}of level $N\in\mathbb Z_{>1}$. The Epstein zeta function on $\varGamma_0(N)$   is defined  as  \cite[Chap.\ II, (2.14)]{GrossZagierI} \begin{align}
E^{\varGamma_0(N)}(z,s):=\sum_{\hat\gamma\in\left.\left(\begin{smallmatrix}*&*\\0&*\end{smallmatrix}\right)\right\backslash\varGamma_0(N)}[\I(\hat\gamma z )]^{s},\quad z\in\mathfrak H,\RE s>1.
\end{align}As pointed out in \cite[Theorem 4.3]{Zhou2025BHS}, part of the Guillera--Rogers theory \cite[(46)]{GuilleraRogers2014}  can be reformulated into the lemma below.\begin{lemma}\label{lm:GR_EZF}For $ z\in\mathfrak H$ satisfying the constraints in \eqref{ineq:z}, along with   $ N=4\sin^2(\nu\pi)\in\{2,3,4\}$ corresponding to  $ \nu\in\left\{ -\frac{1}{4} ,-\frac{1}{3},-\frac{1}{2}\right\}$,  we have\begin{align}
\begin{split}
\RE\mathsf\Sigma_{\nu}^{GR}(z)={}&\frac{8\pi^{2}}{3}\left[E^{\varGamma_0(N)}\left(-\frac{1}{Nz},2\right)- E^{\varGamma_0(N)}(z,2) \right]\\={}&\frac{8\pi^{2}}{3\left( N^{2}-1 \right)}\left[ E^{\varGamma_0(1)}(z,2) -E^{\varGamma_0(1)}(Nz,2)\right].
\end{split}\label{eq:ReEpstein}
\end{align}\qed
\end{lemma} Guillera and Rogers \cite{GuilleraRogers2014} tabulated many closed forms of $\RE\mathsf\Sigma_{\nu}^{GR}(z)$, where both $E^{\varGamma_0(1)}(Nz,2)$ and $E^{\varGamma_0(1)}(z,2)$ were ``solvable'' lattice sums in the sense of Glasser--Zucker \cite{GlasserZucker1980} and  Zucker--Robertson \cite{ZuckerRobertson1984}.

In the rest of this section, we will evaluate more cases of $\RE\mathsf\Sigma_{\nu}^{GR}(z)=\mathsf\Sigma_{\nu}^{GR}(z)$ for $ \RE z=\frac12$, where  $E^{\varGamma_0(1)}(Nz,2)$ and $E^{\varGamma_0(1)}(z,2)$ may not be individually ``solvable'', but the difference  $E^{\varGamma_0(1)}(Nz,2)-E^{\varGamma_0(1)}(z,2)$ can nevertheless be expressed in terms of Dirichlet $L$-values.\subsection{Kronecker's theorem in action}\subsubsection{Discriminants and $j$-invariants\label{subsubsec:disc_j}}Let $ \disc(z)$ be the discriminant for the minimal polynomial of a CM point $z$.
 We normalize Klein's $j$-in\-var\-i\-ant as\begin{align}
j(z)\colonequals{}&\frac{2^{8}\{1-\alpha_{4}(z/2)+[\alpha_{4}(z/2)]^{2}\}^{3}}{[\alpha_{4}(z/2)]^{2}[1-\alpha_{4}(z/2)]^2}.
\end{align}If two CM points $ z_1$ and $z_2$ satisfy $ \disc(z_1)=\disc(z_2)=D$, then $j(z_1)$ and $j(z_2)$ have  the same minimal polynomial as  $ j\left( \frac{D+i\sqrt{-D}}{2} \right)$.

For the purpose of this subsection, we will distinguish two scenarios:\begin{enumerate}[leftmargin=*,
label={(\arabic*{$^\circ$})},ref=(\arabic*{$^\circ$}),
widest=d, align=left]
\item\label{itm:sc1}
The algebraic numbers $j(z)$ and $j(Nz)$ are distinct degree-$4$ irrationals  that share the same minimal polynomial.\item\label{itm:sc2} The algebraic numbers $j(z)$ and $j(Nz)$ both have degrees that do not exceed $2$.
\end{enumerate}

In Tables \ref{tab:ReGRa}--\ref{tab:ReGRc}, the entries corresponding to these two scenarios are separated by  horizontal lines. Within each scenario, the entries in these tables are sorted by descending $\I z$.
 \begin{table}[t]\caption{\label{tab:ReGRa}Special CM points associated with irrational series $ \mathsf\Sigma_{-1/2}^{GR}(z)$ that are expressible via Dirichlet $L$-values}

{\tiny \begin{align*}\begin{array}{@{\,}c@{\,}|@{\,}c@{\,}c@{\,}c|@{\,}l@{\,}}\hline\hline
z-\frac12 & \frac{{}1-2\alpha_{4}(z)}{\I z} & \frac{R_{-1/2}(1-2\alpha_{4}(z))}{\I z} & \frac{16}{\alpha_4(z){[}1-\alpha_{4}(z)]}&\frac{\pi^2\mathsf\Sigma_{-1/2}^{GR}(z)}{2\disc(z)} \\\hline
\vphantom{\frac{\frac\int\int}1}\frac{i}{\sqrt{7}} & \frac{3 \left(17 \sqrt{7}+35\right)}{8 \sqrt{2}} & \frac{35 \sqrt{7}+89}{8 \sqrt{2}} & \begin{smallmatrix}-2^{11} \big(45-17 \sqrt{7}\big)^2\end{smallmatrix} &\begin{smallmatrix}&L_{-56}(2)L_8(2)-L_{-8}(2)L_{56}(2)\hfill\\={}&\frac{\pi^2}{8\sqrt{2}}\left[L_{-56}(2) - \frac{20L_{-8}(2)}{7\sqrt{7}}\right]\hfill\end{smallmatrix}\\[10pt]\frac{i}{\sqrt{22}} & \begin{smallmatrix}66 \left(5 \sqrt{2}+7\right) \end{smallmatrix}&\begin{smallmatrix} 284 \sqrt{2}+401\end{smallmatrix} & \begin{smallmatrix}-\left(2-\sqrt{2}\right)^{12}\end{smallmatrix}&\begin{smallmatrix}&L_{-8}(2)L_{44}(2)-L_{-4}(2)L_{88}(2)\hfill\\={}&\frac{7\pi^{2}}{22 \sqrt{11}}\left[L_{-8}(2) - \frac{23G}{14 \sqrt{2}}\right]\hfill\end{smallmatrix} \\[10pt]
\frac{i}{\sqrt{58}} & \begin{smallmatrix}198 \left(70 \sqrt{29}+377\right) \end{smallmatrix}&\begin{smallmatrix} 3 \left(4234 \sqrt{29}+22801\right)\end{smallmatrix} & \begin{smallmatrix}-\left(5 \sqrt{29}-27\right)^6\end{smallmatrix} &\begin{smallmatrix}&L_{-4}(2)L_{232}(2)-L_{-116}(2)L_{8}(2)\hfill\\={}&\frac{99\pi^{2}}{116 \sqrt{58}}\left[G - \frac{29 \sqrt{29}L_{-116}(2)}{198}\right]\hfill\end{smallmatrix}\\[8pt]
\hline\vphantom{\frac{\frac\int\int}1}
\frac{i}{2\sqrt{7}} & \begin{smallmatrix}6 \left(85 \sqrt{7}+224\right) \end{smallmatrix}& \begin{smallmatrix}32 \left(14 \sqrt{7}+37\right) \end{smallmatrix} & \begin{smallmatrix}\big( 3\sqrt{7} -8\big)^{3}\end{smallmatrix}&\begin{smallmatrix}&\frac{33}{4}\zeta(2)L_{-7}(2)-16L_{-4}(2)L_{28}(2)\hfill\\={}&\frac{11 \pi ^2}{8}\left[L_{-7}(2) - \frac{256G}{77 \sqrt{7}}\right]\hfill\end{smallmatrix} \\
[8pt]
\hline\hline
\end{array}\end{align*}}
\end{table}

\subsubsection{Reductions of lattice sums to Dirichlet $L$-values}
For Scenario \ref{itm:sc1} mentioned  in \S\ref{subsubsec:disc_j}, we will look for CM points $ z_1$ and $z_2$ such that $ j(z)$, $ j(Nz)$, $ j(z_1)$, and $j(z_2)$ are four different degree-$4$ irrationals  that share the same minimal polynomial. Exploiting Kronecker's theorem \cite[Theorem 4]{Siegel1980}, we will represent both\begin{align}
E^{\varGamma_0(1)}(Nz,2)-E^{\varGamma_0(1)}(z,2)+E^{\varGamma_0(1)}(z_{1},2)-E^{\varGamma_0(1)}(z_{2},2)\label{eq:EZF+-+-}
\end{align} and  \begin{align}
E^{\varGamma_0(1)}(Nz,2)-E^{\varGamma_0(1)}(z,2)-E^{\varGamma_0(1)}(z_{1},2)+E^{\varGamma_0(1)}(z_{2},2)\label{eq:EZF+--+}
\end{align}through special Dirichlet $L$-values, from which we can deduce the closed form of $E^{\varGamma_0(1)}(Nz,2)-E^{\varGamma_0(1)}(z,2)$. 

Concretely speaking, if $ j(w_1),\dots,j(w_h)$ are distinct algebraic numbers sharing the same minimal polynomial of degree $h$,  and $\disc(w_{1})=\dots=\disc(w_h)=D\notin\{-3,-4\}$ is a fundamental discriminant\footnote{This means that either (A) or (B) is true, where (A)~$ D\neq1$, $ D\equiv1\pmod4$, and $D$ is square-free; (B)~$D\equiv 8\text{ or }12\pmod{16} $ and $ D/4 $ is square-free. } that  factorizes into the product of two fundamental discriminants  $ d_1$ and $d_2$, then Kronecker's theorem \cite[Theorem 4]{Siegel1980} assumes the following form (cf.\ \cite[(2.3)]{ZuckerRobertson1984}):\footnote{Here, the normalizing factor $ -\frac{d_1d_2}{4\zeta(4)}=-\frac{2D}{45\pi^{4}}$ converts the Dedekind zeta function (appearing in the statement of \cite[Theorem 4]{Siegel1980}) to the Epstein zeta function.}\begin{align}
\sum_{j=1}^h\chi _{d_{1},d_2}(w_j)E^{\varGamma_0(1)}(w_{j},2)={}&-\frac{d_{1}d_2}{4\zeta(4)} L_{d_1}(2)L_{d_2}(2),
\label{eq:fundKronecker}
\end{align}where the character $\chi _{d_{1},d_2}(w_j)$ is explicitly computable from the Kronecker symbol involving the norm of the integral ideal associated with $w_j$. Later on, we will refer to this as a ``fundamental application'' of Kronecker's theorem.

For ``non-fundamental applications'' of Kronecker's theorem, one sometimes needs to twist the right-hand side of \eqref{eq:fundKronecker} with additional factors, as done by Glasser--Zucker \cite{GlasserZucker1980} and  Zucker--Robertson \cite{ZuckerRobertson1984}.

For Scenario \ref{itm:sc2} mentioned  in \S\ref{subsubsec:disc_j}, we will mostly fall back on tabulated entries in \cite{GlasserZucker1980}, some of which were consequences of  ``non-fundamental applications'' as well. To make full use of existent tables, we will also exploit  \eqref{eq:fundKronecker} in the case where $h=2$.    \subsubsection{Proof of Theorem \ref{thm:2k3k4kL2}\ref{itm:2kL2}} None of the following numbers constitutes a fundamental discriminant:\begin{align}
\disc\left( \frac{1}{2}+\frac{ i}{\sqrt{7}}\right)={}&-448=-2^{6}\cdot7,\\\disc\left( \frac{1}{2}+\frac{ i}{\sqrt{22}}\right)={}&-352=-2^5\cdot11,\\\disc\left( \frac{1}{2}+\frac{ i}{\sqrt{58}}\right)={}&-928=-2^{5}\cdot29.
\end{align} Nevertheless, Kronecker's theorem works without requiring additional twist factors for these three cases. To flesh out, with \begin{align}
z_{1}=\frac{1}{2}+\frac{\sqrt{-\disc(z)/4}i}{4}=\begin{cases}\frac12+\sqrt{7}i & \text{if }z=\frac{1}{2}+\frac{ i}{\sqrt{7}}, \\
 \frac12+\frac{\sqrt{22}i}{2}& \text{if }z=\frac{1}{2}+\frac{ i}{\sqrt{22}}, \\ \frac12+\frac{\sqrt{58}i}{2}&\text{if }z=\frac{1}{2}+\frac{i}{\sqrt{58}},
\end{cases}
\end{align}and  \begin{align}
z_{2}=\sqrt{-\disc(z)/4}i=\begin{cases}4\sqrt{7}i & \text{if }z=\frac{1}{2}+\frac{ i}{\sqrt{7}}, \\
2\sqrt{22}i & \text{if }z=\frac{1}{2}+\frac{ i}{\sqrt{22}}, \\2\sqrt{58}i&\text{if }z=\frac{1}{2}+\frac{i}{\sqrt{58}},
\end{cases}
\end{align} we obtain\begin{align}
\begin{split}
&-\frac{4\zeta(4)}{\disc(z)}\left[ E^{\varGamma_0(1)}(4z,2)- E^{\varGamma_0(1)}(z,2)-E^{\varGamma_0(1)}(z_{1},2)+E^{\varGamma_0(1)}(z_{2},2)\right]
\\={}&\begin{cases}L_{-56}(2)L_{8}(2) & \text{if }z=\frac{1}{2}+\frac{ i}{\sqrt{7}}, \\
L_{-8}(2)L_{44}(2) & \text{if }z=\frac{1}{2}+\frac{ i}{\sqrt{22}}, \\
L_{-4}(2)L_{232}(2) & \text{if }z=\frac{1}{2}+\frac{i}{\sqrt{58}}, \\
\end{cases}\end{split}
\end{align}and\begin{align}
\begin{split}
&\frac{4\zeta(4)}{\disc(z)}\left[ E^{\varGamma_0(1)}(4z,2)- E^{\varGamma_0(1)}(z,2)+E^{\varGamma_0(1)}(z_{1},2)-E^{\varGamma_0(1)}(z_{2},2)\right]
\\={}&\begin{cases}L_{-8}(2)L_{56}(2) & \text{if }z=\frac{1}{2}+\frac{ i}{\sqrt{7}}, \\
L_{-4}(2)L_{88}(2) & \text{if }z=\frac{1}{2}+\frac{ i}{\sqrt{22}}, \\
L_{-116}(2)L_{8}(2) & \text{if }z=\frac{1}{2}+\frac{i}{\sqrt{58}}, \\
\end{cases}\end{split}
\end{align} which verify the first three entries in Table \ref{tab:ReGRa} as well as the first three formulae in Theorem \ref{thm:2k3k4kL2}\ref{itm:2kL2}.

The last entry in Table \ref{tab:ReGRa} [corresponding to the last formula in Theorem \ref{thm:2k3k4kL2}] deserves a little more explanation. For $ z=\frac{1}{2}+\frac{i}{2\sqrt{7}}$, we have \begin{align}
\disc(z)=-7\quad\text{and}\quad \disc(4z)=-112=-2^{4}\cdot7.
\end{align}Here, $ j(z)=-3^3\cdot5^3=j\left( \frac{1+\sqrt{7}i}{2} \right)$ is an integer, so one may read off the corresponding special value\begin{align}
E^{\varGamma_0(1)}\left(\frac{1+\sqrt{7}i}{2},2\right)=\frac{7\zeta(2)L_{-7}(2)}{4\zeta(4)}=\frac{105}{4\pi^2}L_{-7}(2)
\end{align} from \cite[Table VI]{GlasserZucker1980}. Since both $ j(\cdot)$ and $ E^{\varGamma_0(1)}(\cdot,2)$ are $ \varGamma_0(1)$-invariants, we have\begin{align}
E^{\varGamma_0(1)}\left(z,2\right)=E^{\varGamma_0(1)}\left(\frac{1+\sqrt{7}i}{2},2\right)=\frac{105}{4\pi^2}L_{-7}(2).
\end{align}Meanwhile, $ j(4z)=(32520-12285 \sqrt{7})^3$ is a quadratic irrational, whose Galois conjugate is $ j(2\sqrt{7}i)=(32520+12285 \sqrt{7})^3$. Looking up \begin{align}
E^{\varGamma_0(1)}\big(2\sqrt{7}i,2\big)=\frac{56}{4\zeta(4)}\left[\frac{41}{64} \zeta(2) L_{-7}(2)+L_{-4}(2)L_{28}(2)\right]\label{eq:a-112}
\end{align}from  \cite[Table VI]{GlasserZucker1980},  and applying \cite[Theorem 4]{Siegel1980} to\begin{align}\begin{split}&
E^{\varGamma_0(1)}\left(4z,2\right)-
E^{\varGamma_0(1)}\big(2\sqrt{7}i,2\big)=-\frac{112}{4\zeta(4)}L_{-4}(2)L_{28}(2),\label{eq:112minus}
\end{split}
\end{align}we find\begin{align}\begin{split}
E^{\varGamma_0(1)}\left(4z,2\right)={}&\frac{56}{4\zeta(4)}\left[\frac{41}{64} \zeta(2) L_{-7}(2)-L_{-4}(2)L_{28}(2)\right]\\={}&\frac{15 \left[2009 L_{-7}(2)-768 \sqrt{7} G\right]}{224 \pi ^2},\label{eq:b-112}
\end{split}
\end{align} which completes the evaluation of $E^{\varGamma_0(1)}\left(4z,2\right)-E^{\varGamma_0(1)}\left(z,2\right)$.

Before moving on, we remark that \eqref{eq:112minus} is a special case of \eqref{eq:fundKronecker} without additional twists, unlike [cf.\ \eqref{eq:a-112} and \eqref{eq:b-112}]\begin{align}
E^{\varGamma_0(1)}\left(4z,2\right)+
E^{\varGamma_0(1)}\big(2\sqrt{7}i,2\big)=\frac{41}{64}\frac{112}{4\zeta(4)}\zeta(2) L_{-7}(2)
\end{align} with an  extra rational factor $\frac{41}{64} $  attributable to the non-fundamental discriminant $ \disc(4z)=-112$. Here, one notes that if $ j(w_1),\dots,j(w_h)$ are distinct algebraic numbers sharing the same minimal polynomial of degree $h$,  and $\disc(w_{1})=\dots=\disc(w_h)=D\notin\{-3,-4\}$ is a fundamental discriminant, then Dirichlet's theorem \cite[(2.1)]{ZuckerRobertson1984}\begin{align}
\sum_{j=1}^hE^{\varGamma_0(1)}(w_{j},2)={}&-\frac{D}{4\zeta(4)} \zeta(2)L_{D}(2)\
\end{align} serves as a companion to \eqref{eq:fundKronecker}. \subsubsection{Proof of Theorem  \ref{thm:2k3k4kL2}\ref{itm:3kL2}}
\begin{table}[t]\caption{\label{tab:ReGRb}Special CM points associated with irrational series $ \mathsf\Sigma_{-1/3}^{GR}(z)$ that are expressible via Dirichlet $L$-values} {\tiny \begin{align*}\begin{array}{@{\,}c@{\,}|@{\,}c@{\,}c@{\,}c|@{\,}l@{\,}}\hline\hline
z-\frac12 & {}\frac{1-2\alpha_{3}(z)}{\I z} & \frac{R_{-1/3}(1-2\alpha_{3}(z))}{\I z} & \frac{27}{\alpha_3(z){[}1-\alpha_{3}(z)]}&\frac{4\pi^2\mathsf\Sigma_{-1/3}^{GR}(z)}{15\disc(z)} \\\hline\vphantom{\frac{\frac\int\int}1}\frac{\sqrt{555}i}{{30}} & \begin{smallmatrix}\frac{15 \left(2898 \sqrt{37}+3145\right)}{98494\sqrt{3}} \end{smallmatrix} & \begin{smallmatrix}\frac{6438 \sqrt{37}+30355}{49247\sqrt{3}} \end{smallmatrix} & \begin{smallmatrix}-2^6 3^3 \left(\frac{3 \sqrt{37}-17}{2}\right)^6 \end{smallmatrix} &\begin{smallmatrix}&L_{-111}(2)L_{5}(2)-L_{-3}(2)L_{185}(2)\hfill\\={}&\frac{4 \pi ^2}{25 \sqrt{5}}\left[L_{-111}(2)-\frac{380K}{37 \sqrt{37}}\right]\hfill\end{smallmatrix}\\
[8pt]
\frac{\sqrt{435}i}{{30}} & \begin{smallmatrix}&\frac{15}{2}\\&\times\frac{323 \sqrt{145}+13195}{12528 \sqrt{15}} \end{smallmatrix} & \begin{smallmatrix} \frac{2407 \sqrt{145}+35375}{12528 \sqrt{15}}\end{smallmatrix} & \begin{smallmatrix} &-\frac{432}{5} \\&\times\left(2975-247 \sqrt{145}\right)^2 \end{smallmatrix}&\begin{smallmatrix}&L_{-15}(2)L_{29}(2)-L_{-87}(2)L_{5}(2)\hfill\\={}&\frac{12\pi^{2}}{29 \sqrt{29}}\left[ L_{-15}(2)-\frac{29\sqrt{145}L_{-87}(2)}{375} \right]\hfill\end{smallmatrix} \\\frac{i}{\sqrt{3}} & \begin{smallmatrix}\frac{9\left(11 \sqrt{6}+51\right)}{125}  \end{smallmatrix} & \begin{smallmatrix}\frac{4\left(19 \sqrt{6}+54\right)}{125}  \end{smallmatrix} & \begin{smallmatrix} -\frac{3^3 \left(\sqrt{6}+2\right) \left(\sqrt{6}-1\right)^6}{\left(2 \sqrt{6}+5\right)^2}\end{smallmatrix} &\begin{smallmatrix}&\frac{9}{8}L_{-4}(2)L_{48}(2)-L_{-24}(2)L_{8}(2)\hfill\\={}&\frac{\sqrt{3}\pi^2}{16}\left[ G-\sqrt{\frac{2}{3}} L_{-24}(2)\right]\hfill\end{smallmatrix} \\\frac{\sqrt{195}i}{30} & \begin{smallmatrix} \frac{15 \left(9 \sqrt{13}+26\right)}{26 \sqrt{3}}\end{smallmatrix} & \begin{smallmatrix}\frac{39 \sqrt{13}+134}{13 \sqrt{3}} \end{smallmatrix} & \begin{smallmatrix}-2^6 3^3 \left(\frac{\sqrt{13}-3}{2}\right)^6 \end{smallmatrix} &\begin{smallmatrix}&L_{-39}(2)L_{5}(2)-L_{-3}(2)L_{65}(2)\hfill\\={}&\frac{4 \pi ^2}{25 \sqrt{5}}\left[L_{-39}(2)-\frac{80K}{13 \sqrt{13}}\right]\hfill\end{smallmatrix}\\
\frac{i}{\sqrt{6}} & \begin{smallmatrix} 17 \sqrt{3}+27\end{smallmatrix} & \begin{smallmatrix} \frac{2\left(16 \sqrt{3}+27\right)}{3} \end{smallmatrix} & \begin{smallmatrix} -\frac{2\cdot 3^3 \left(\sqrt{3}-1\right)}{\left(\sqrt{3}+2\right)^4}\end{smallmatrix}&\begin{smallmatrix}&\frac{11}8L_{-3}(2)L_{32}(2)-L_{-4}(2)L_{24}(2)\hfill\\={}&\frac{11\pi^{2}}{64 \sqrt{2}}\left( K-\frac{16G}{11 \sqrt{3}} \right)\hfill\end{smallmatrix} \\
[8pt]
\hline\vphantom{\frac{\frac\int\int}1}\frac{\sqrt{11}i}{6} & \begin{smallmatrix}\frac{3 \left(27 \sqrt{33}+91\right)}{32 \sqrt{11}} \end{smallmatrix} & \begin{smallmatrix}\frac{9 \left(25 \sqrt{33}+121\right)}{176 \sqrt{11}} \end{smallmatrix} & \begin{smallmatrix}32 \left(91 \sqrt{33}-523\right) \end{smallmatrix}&\begin{smallmatrix}&L_{-3}(2)L_{33}(2)-\frac{16}{27}\zeta (2)L_{-11}(2)\hfill\\={}&\frac{8\pi^{2}}{11 \sqrt{33}}\left[ K-\frac{11}{27}  \sqrt{\frac{11}{3}}L_{-11}(2)\right]\hfill\end{smallmatrix} \\
\frac{i}{2} & \begin{smallmatrix}3 \sqrt{3}+7 \end{smallmatrix} & \begin{smallmatrix}2 \left(\sqrt{3}+2\right) \end{smallmatrix} & \begin{smallmatrix} -24 \left(7 \sqrt{3}-12\right)\end{smallmatrix} &\begin{smallmatrix}&\frac{16}{3}\zeta (2)L_{-4}(2)-9L_{-3}(2)L_{12}(2)\hfill\\={}&\frac{8 \pi ^2}{9}\left( G-\frac{9 \sqrt{3}K}{16} \right)\hfill\end{smallmatrix}\\
[8pt]\hline\hline
\end{array}\end{align*}}
\end{table}

The following discriminants are all fundamental:\begin{align}
\disc\left( \frac{1}{2}+\frac{\sqrt{555}i}{{30}} \right)={}&-555=-3\cdot5\cdot37,\\\disc\left( \frac{1}{2}+\frac{\sqrt{435}i}{{30}} \right)={}&-435=-3\cdot5\cdot29,\\\disc\left( \frac{1}{2}+\frac{\sqrt{195}i}{{30}} \right)={}&-195=-3\cdot5\cdot13.
\end{align}The same statement holds true for the members of \begin{align}\{-111,5,-3,185,-15,29,-87,-39,65\}.\end{align}Therefore, the three corresponding entries in Table \ref{tab:ReGRb} can be handled with  ``fundamental applications'' of Kronecker's theorem.  To see this, we introduce $ z_1=5z$ and $z_2=15z$, 
so it follows that \begin{align}
\begin{split}
&-\frac{4\zeta(4)}{\disc(z)}\left[ E^{\varGamma_0(1)}(3z,2)- E^{\varGamma_0(1)}(z,2)-E^{\varGamma_0(1)}(z_{1},2)+E^{\varGamma_0(1)}(z_{2},2)\right]
\\={}{}&\begin{cases}L_{-111}(2)L_{5}(2) & \text{if }z=\frac{1}{2}+\frac{\sqrt{555}i}{{30}}, \\
L_{-15}(2)L_{29}(2) & \text{if }z=\frac{1}{2}+\frac{\sqrt{435}i}{{30}}, \\
 L_{-39}(2)L_{5}(2)& \text{if }z=\frac{1}{2}+\frac{\sqrt{195}i}{{30}}, \\
\end{cases}\end{split}
\end{align}and \begin{align}
\begin{split}
&\frac{4\zeta(4)}{\disc(z)}\left[ E^{\varGamma_0(1)}(3z,2)- E^{\varGamma_0(1)}(z,2)+E^{\varGamma_0(1)}(z_{1},2)-E^{\varGamma_0(1)}(z_{2},2)\right]
\\={}{}&\begin{cases}L_{-3}(2)L_{185}(2) & \text{if }z=\frac{1}{2}+\frac{\sqrt{555}i}{{30}}, \\
 L_{-87}(2)L_{5}(2)& \text{if }z=\frac{1}{2}+\frac{\sqrt{435}i}{{30}}, \\
 L_{-3}(2)L_{65}(2)& \text{if }z=\frac{1}{2}+\frac{\sqrt{195}i}{{30}}, \\
\end{cases}\end{split}
\end{align}  which together entail the corresponding entries in the last column of Table \ref{tab:ReGRb}. 

In  Table \ref{tab:ReGRb}, we also have two trickier cases of Scenario \ref{itm:sc1}, with non-fundamental discriminants:\begin{align}
\disc\left( \frac12+\frac{i}{\sqrt{3}} \right)={}&-192=-2^{6}\cdot3,
\\
\disc\left( \frac12+\frac{i}{\sqrt{6}} \right)={}&-96=-2^5\cdot3.\end{align}Now we switch to a new setting, where $z_1=4z$ and \begin{align}
z_2=\frac{\sqrt{-\disc(z)}i}{2}={}&\begin{cases}4\sqrt{3}i & \text{if }z=\frac{1}{2}+\frac{i}{\sqrt{3}}, \\
2\sqrt{6}i & \text{if }z=\frac{1}{2}+\frac{i}{\sqrt{6}}. 
 \\
\end{cases}
\end{align}We have \begin{align}
\begin{split}
&-\frac{4\zeta(4)}{\disc(z)}\left[ E^{\varGamma_0(1)}(3z,2)- E^{\varGamma_0(1)}(z,2)-E^{\varGamma_0(1)}(z_{1},2)+E^{\varGamma_0(1)}(z_{2},2)\right]
\\={}{}&\begin{cases}\frac{9}{8}L_{-4}(2)L_{48}(2) & \text{if }z=\frac{1}{2}+\frac{i}{\sqrt{3}}, \\
 \frac{11}8L_{-3}(2)L_{32}(2)& \text{if }z=\frac{1}{2}+\frac{i}{\sqrt{6}}, \\
\end{cases}\end{split}
\end{align}and \begin{align}
\begin{split}
&\frac{4\zeta(4)}{\disc(z)}\left[ E^{\varGamma_0(1)}(3z,2)- E^{\varGamma_0(1)}(z,2)+E^{\varGamma_0(1)}(z_{1},2)-E^{\varGamma_0(1)}(z_{2},2)\right]
\\={}{}&\begin{cases}L_{-24}(2)L_{8}(2) & \text{if }z=\frac{1}{2}+\frac{i}{\sqrt{3}}, \\
L_{-4}(2)L_{24}(2) & \text{if }z=\frac{1}{2}+\frac{i}{\sqrt{6}}, \\
\end{cases}\end{split}
\end{align}which involve additional twist factors (cf.\ \cite{GlasserZucker1980,ZuckerRobertson1984}).
 
 There are two entries in  Table \ref{tab:ReGRb} corresponding to Scenario \ref{itm:sc2}, both of which echo back to the last entry in Table \ref{tab:ReGRa}. We will elaborate on these entries in the next two paragraphs.

For $ z=\frac12+\frac{\sqrt{11}i}{6}$, we have \begin{itemize}
\item 
an integer $ j(3z)=-2^{15}=j\left( \frac{1+\sqrt{11}i}{2} \right)$ as well as a tabulated value    \cite[Table VI]{GlasserZucker1980}\begin{align}
E^{\varGamma_0(1)}(3z,2)=E^{\varGamma_0(1)}\left(\frac{1+\sqrt{11}i}{2},2\right)=\frac{165L_{-11}(2)}{4 \pi ^2}.
\end{align}\item a quadratic irrational \begin{align}
j(z)= \frac{\left(\frac{\sqrt{33}-5}{2}\right)^{15} \left(\frac{\sqrt{33}+5}{2}\right)^{18} \left(2 \sqrt{33}+11\right)^3 \left(\sqrt{33}-4\right)^3}{\left(23+4 \sqrt{33}\right)^5}
\end{align} that is conjugate to \begin{align}\begin{split}
&j\left( \frac{1+3\sqrt{11}i}{2} \right)\\={}&\frac{\left(\frac{\sqrt{33}+5}{2}\right)^{15} \left(\frac{\sqrt{33}-5}{2}\right)^{18} \left(11-2 \sqrt{33}\right)^3 \left(\sqrt{33}+4\right)^3}{\left(23-4 \sqrt{33}\right)^5},
\end{split}\end{align} along with a tabulated value    \cite[Table VI]{GlasserZucker1980}\begin{align}
\begin{split}&
E^{\varGamma_0(1)}\left(\frac{1+3\sqrt{11}i}{2},2\right)\\={}&\frac{99}{8 \zeta (4)}\left[ \frac{22\zeta(2)L_{-11}(2)}{27} +L_{-3}(2)L_{33}(2)\right].
\end{split}
\end{align}
\end{itemize}Combining the information above with an application of \eqref{eq:fundKronecker} for $h=2$:
\begin{align}
E^{\varGamma_0(1)}\left(\frac{1+3\sqrt{11}i}{2},2\right)-E^{\varGamma_0(1)}\left(z,2\right)=\frac{99L_{-3}(2)L_{33}(2)}{4\zeta(4)},
\end{align}we arrive at our claimed result.

For $ z=\frac{1+i}2$, we have\begin{itemize}
\item an integer $ j(z)=2^{6}3^{3}=j\left( i \right)$ 
along with a tabulated value \cite[Table VI]{GlasserZucker1980}\begin{align}
E^{\varGamma_0(1)}\left(z,2\right)=E^{\varGamma_0(1)}\left(i,2\right)=\frac{30L_{-4}(2)}{\pi^{2}}.
\end{align}\item a quadratic irrational $ j(3z)= \frac{\left(2-3 \sqrt{3}\right)^3 \left(2 \sqrt{3}-1\right)^3 \left( \sqrt{3} \right)^3\left(\sqrt{3}-1\right)^{12}}{\left(\sqrt{3}-2\right)^4}$ that is conjugate to $ j(3i)=\frac{\left(2+3 \sqrt{3}\right)^3 \left(2 \sqrt{3}+1\right)^3  \left( \sqrt{3} \right)^3 \left(\sqrt{3}+1\right)^{12}}{\left(\sqrt{3}+2\right)^4}$, as well as a tabulated value \cite[Table VI]{GlasserZucker1980}\begin{align}
E^{\varGamma_0(1)}\left(3i,2\right)=\frac{18}{4\zeta(4)}\left[ \frac{28\zeta(2)L_{-4}(2)}{27} +L_{-3}(2)L_{12}(2)\right].
\end{align}
\end{itemize}Again, with \eqref{eq:fundKronecker} for $h=2$, we have \begin{align}
E^{\varGamma_0(1)}\left(3i,2\right)-E^{\varGamma_0(1)}\left(3z,2\right)=\frac{36L_{-3}(2)L_{12}(2)}{4\zeta(4)},
\end{align}which will conclude our verification for Table \ref{tab:ReGRb} and Theorem \ref{thm:2k3k4kL2}\ref{itm:3kL2}.
 
\subsubsection{Proof of Theorem  \ref{thm:2k3k4kL2}\ref{itm:4kL2}}\begin{table}[t]\caption{\label{tab:ReGRc}Special CM points associated with irrational series $ \mathsf\Sigma_{-1/4}^{GR}(z)$ that are expressible via Dirichlet $L$-values} {\tiny \begin{align*}\begin{array}{@{\,}c@{\,}|@{\,}c@{\,}c@{\,}c|@{\,}l@{\,}}\hline\hline
z-\frac12 & \frac{{}1-2\alpha_{2}(z)}{\I z} & \frac{R_{-1/4}(1-2\alpha_{2}(z))}{\I z} & \frac{64}{\alpha_2(z){[}1-\alpha_{2}(z)]}&\frac{\pi^2\mathsf\Sigma_{-1/4}^{GR}(z) }{10\disc(z)}\\\hline
\vphantom{\frac{\frac\int\int}1}\frac{\sqrt{85}i}{10} & \begin{smallmatrix}\frac{35 \left(92 \sqrt{17}+2091\right)}{55233} \end{smallmatrix} & \begin{smallmatrix} \frac{6868 \sqrt{17}+36591}{110466}\end{smallmatrix} & \begin{smallmatrix} -2^{10} 3^4 \left(41-10 \sqrt{17}\right)^4\end{smallmatrix} &\begin{smallmatrix}&L_{-4}(2)L_{85}(2)-L_{-68}(2)L_{5}(2)\hfill\\={}&\frac{72\pi^{2}}{85 \sqrt{85}}\left[G-\frac{17 \sqrt{17}L_{-68}(2)}{90}\right]\hfill\end{smallmatrix}\\ [8pt]
\frac{\sqrt{253}i}{22} & \begin{smallmatrix}\frac{65 \left(11092 \sqrt{11}+19437\right)}{815409} \end{smallmatrix} & \begin{smallmatrix} &\frac{11937508 \sqrt{11}}{37508814}\\&+\frac{37515813}{37508814}\end{smallmatrix} & \begin{smallmatrix} &-2^{10} 3^4 7^4\\&\times \left(1121-338 \sqrt{11}\right)^4\end{smallmatrix}  &\begin{smallmatrix}&L_{-11}(2)L_{92}(2)-L_{-4}(2)L_{253}(2)\hfill\\={}&\frac{10\pi^{2}}{23 \sqrt{23}}\left[L_{-11}(2)-\frac{ 36G}{11\sqrt{11}}\right]\hfill\end{smallmatrix}\\
[8pt]
\hline
\vphantom{\frac{\frac{\frac11}\int}1}\frac{i}{\sqrt{2}} & \begin{smallmatrix}\frac{5\left(13 \sqrt{2}+32\right)}{49}  \end{smallmatrix} & \begin{smallmatrix} \frac{6\left(6 \sqrt{2}+11\right)}{49} \end{smallmatrix} & \begin{smallmatrix}-\frac{2^3 \left(3-\sqrt{2}\right)^4}{\sqrt{2}+1} \end{smallmatrix} &\begin{smallmatrix}&L_{-4}(2)L_{8}(2)-\frac{3}{8}\zeta(2)L_{-8}(2)\hfill\\={}&\frac{\pi^2}{8\sqrt{2}}\left[G-\frac{L_{-8}(2)}{\sqrt{2}}\right]\hfill\end{smallmatrix}\\
[8pt]
\hline\hline
\end{array}\end{align*}}
\end{table}
Both \begin{align}
\disc\left( \frac{1}{2}+\frac{ \sqrt{85}i}{10}\right)={}&-340=-2^2\cdot5\cdot17\intertext{and}\disc\left( \frac{1}{2}+\frac{ \sqrt{253}i}{22} \right)={}&-1012=-2^2\cdot11\cdot23
\end{align}are fundamental discriminants, and so are the numbers in the set $ \{-4,85,-68,5,-11,92,253\}$. Indeed, the first two entries of Table \ref{tab:ReGRc} are both tractable by ``fundamental applications'' of Kronecker's theorem. Setting  \begin{align}
z_{1}=\frac{1+\sqrt{-\disc(z)/4}i}{2}=\begin{cases}\frac{1+\sqrt{85}i}{2} & \text{if }z=\frac{1}{2}+\frac{ \sqrt{85}i}{10}, \\
\frac{1+\sqrt{253}i}{2} & \text{if }z=\frac{1}{2}+\frac{ \sqrt{253}i}{22}, \\
\end{cases}
\end{align}and \begin{align}
z_{2}=\sqrt{-\disc(z)/4}i=\begin{cases}\sqrt{85}i & \text{if }z=\frac{1}{2}+\frac{ \sqrt{85}i}{10}, \\
\sqrt{253}i & \text{if }z=\frac{1}{2}+\frac{ \sqrt{253}i}{22}, \\
\end{cases}
\end{align} one gets\begin{align}
\begin{split}&
-\frac{4\zeta(4)}{\disc(z)}\left[ E^{\varGamma_0(1)}(2z,2)- E^{\varGamma_0(1)}(z,2)-E^{\varGamma_0(1)}(z_{1},2)+E^{\varGamma_0(1)}(z_{2},2)\right]
\\={}&\begin{cases}L_{-4}(2)L_{85}(2) & \text{if }z=\frac{1}{2}+\frac{ \sqrt{85}i}{10}, \\
L_{-11}(2)L_{92}(2) & \text{if }z=\frac{1}{2}+\frac{ \sqrt{253}i}{22}, \\
\end{cases}\end{split}
\end{align}and
\begin{align}
\begin{split}&
\frac{4\zeta(4)}{\disc(z)}\left[ E^{\varGamma_0(1)}(2z,2)- E^{\varGamma_0(1)}(z,2)+E^{\varGamma_0(1)}(z_{1},2)-E^{\varGamma_0(1)}(z_{2},2)\right]
\\={}&\begin{cases}L_{-68}(2)L_{5}(2) & \text{if }z=\frac{1}{2}+\frac{ \sqrt{85}i}{10}, \\
L_{-4}(2)L_{253}(2) & \text{if }z=\frac{1}{2}+\frac{ \sqrt{253}i}{22}, \\
\end{cases}\end{split}
\end{align}thereby  confirming \eqref{eq:d=-340} and \eqref{eq:d=-1012}, respectively.

The last entry of Table \ref{tab:ReGRc} hearkens back to its counterpart in Table \ref{tab:ReGRa}.
Here, $ j(2z)=2^6\cdot5^3=j\left( \sqrt{2} i\right)$ is an integer, so  \cite[Table VI]{GlasserZucker1980} entails\begin{align}
E^{\varGamma_0(1)}(2z,2)=E^{\varGamma_0(1)}(\sqrt{2}i,2)=\frac{30L_{-8}(2)}{\pi^{2}}.
\end{align}In the meantime, $ j(z)=\left(190-130 \sqrt{2}\right)^3$ is a quadratic irrational, with Galois conjugate $ j(2\sqrt{2}i)=\left(190+130 \sqrt{2}\right)^3$.
Kronecker's theorem for $h=2$ brings us \begin{align}
E^{\varGamma_0(1)}(z,2)-E^{\varGamma_0(1)}(2\sqrt{2}i,2)=-\frac{32}{4\zeta(4)}L_{-4}(2)L_8(2),
\end{align}while  \cite[Table VI]{GlasserZucker1980}  allows us to compute\begin{flalign}E^{\varGamma_0(1)}(2\sqrt{2}i,2)=\frac{4}{\zeta(4)}\left[ \frac{7\zeta(2)L_{-8}(2)}{8} +L_{-4}(2)L_8(2)\right].\end{flalign} Now the proof is complete.

\end{document}